\documentclass{amsart}
\usepackage{amssymb,amsmath,amsthm}
\usepackage{color}
\usepackage{fullpage}
\usepackage{graphicx}

\renewcommand{\preceq}{\preccurlyeq}

\newcommand{\supp}{\operatorname{supp}}

\newcommand{\SOP}{\operatorname{SOP}}
\newcommand{\IP}{\operatorname{IP}}
\newcommand{\TP}{\operatorname{TP}_2}

\newtheorem{theorem}{Theorem}[section]
\newtheorem{lemma}[theorem]{Lemma}
\newtheorem{corollary}[theorem]{Corollary}
\newtheorem{proposition}[theorem]{Proposition}

\theoremstyle{definition}
\newtheorem*{example}{Example}

\newcommand{\bigO}{\mathcal{O}}
\newcommand{\smallo}{{\scriptstyle\mathcal{O}}}
\newcommand{\smallk}{{\mathbf{k}}}

\newcommand{\N}{\mathbb{N}}

\newcommand{\fM}{\mathfrak{M}}

\newcommand{\fm}{\mathfrak{m}}
\newcommand{\fn}{\mathfrak{n}}

\newcommand{\cM}{\mathcal{M}}
\newcommand{\cP}{\mathcal{P}}

\title{Truncation in Hahn Fields is Undecidable and Wild}
\author{Santiago Camacho}

\address{Department of Mathematics, University of Illinois at Urbana-Champaign, 1409 West Green Street,
Urbana}
\email{scamach2@illinois.edu}
\date{}

\begin{document}

\begin{abstract}
We show that in any nontrivial Hahn field with truncation as a primitive operation we can interpret the monadic second-order logic of the additive monoid of natural numbers. We also specify a definable binary relation on such a structure that has $\SOP$ and $\TP$.
\end{abstract}
\maketitle

\section{Introduction and Notations}

Generalized series have been used in the past few decades in order to generalize classical asymptotic series expansions such as Laurent series and Puiseux series. Certain generalized series fields, such as the field of Logarithmic-Exponential Series \cite{DMM}, provide for a richer ambient structure, due to the fact that these series are closed under many common algebraic and analytic operations. In the context of generalized series the notion of truncation becomes an interesting subject of study. In the classical cases, a proper truncation of any Laurent series $\sum_{k\geq k_0}{r_kx^k}$ amounts to a polynomial in the variables $x$ and $x^{-1}$. In the general setting a proper truncation of an infinite series can still be an infinite series. It has been shown by various authors \cite{D,F,MR} that truncation is a robust notion, in the sense that certain natural extensions of truncation closed sets and rings remain truncation closed. We here look at some first-order model theoretic properties of the theory of a Hahn Field equipped with truncation. We show that such theories are very wild in the sense that they can even interpret the theory of $(\N;+,\times)$ via the interpretation of $(\N, \mathcal{P}(\N); +,\in)$, and are thus undecidable. We also indicate definable binary relations with ``bad'' properties such as the strict order property and the tree property of the second kind. The author would like to thank Philipp Hieronymi and Erik Walsberg for bringing the bibliography of monadic second-order logic to his attention. 

\noindent
\subsection*{Notations}
We let $m$ and $n$ range over $\N=\{0,1,\ldots\}$.
For a set $S$, we denote its powerset by $\cP(S)$.
Given a set $S$, and a tuple of variables $x$ we write $S^x$ for the cartesian product $S^{|x|}$ where $|x|$ denotes the length of the tuple $x$. 
Given a language $L$, an $L$-formula $\phi(x)$, and an $L$-structure $\cM= (M;\ldots)$ we let $\phi(M^x)$ denote the set $\{a \in M^x: \cM\models \phi(a)\}$. For a field $K$ we let $K^\times =K\setminus\{0\}$.
\section{Preliminaries}

\subsection*{Hahn Series}
Let $\Gamma$ be an additive ordered abelian group. 
Let $\smallk$ be a field. We indicate a function $f:\Gamma\rightarrow \smallk$ suggestively as a series $f=\sum_{\gamma\in \Gamma}f_\gamma t^\gamma$ where $f_\gamma = f(\gamma)$ and $t$ is a symbol, and let $\supp(f) := \{\gamma\in \Gamma: f_\gamma \neq 0\}$ be the support of $f$. 
We denote the Hahn series field over $\smallk$ with value group $\Gamma$ by 
\[\smallk((t^\Gamma)):= \left\{f=\sum_{\gamma\in \Gamma}f_\gamma t^\gamma: \supp(f) \text{ is well-ordered}\right\},\]
equipped with the usual operations of addition, and multiplication, that is with $\alpha,\beta,\gamma $ ranging over $\Gamma$:
\[f+g=\left(\sum_{\gamma}f_\gamma t^\gamma\right)+\left(\sum_{\gamma}g_\gamma t^\gamma\right) = \sum_{\gamma}\left(f_\gamma +g_\gamma\right)t^\gamma,\]
\[fg = \left(\sum_{\gamma}f_\gamma t^\gamma\right)\left(\sum_{\gamma}g_\gamma t^\gamma\right) = \sum_{\gamma} \left(\sum_{\alpha + \beta = \gamma} f_\alpha g_\beta \right)t^\gamma.\]

Let $f=\sum_{\gamma}f_\gamma t^\gamma$ be in $\smallk((t^\Gamma))$ and $\delta\in \Gamma$. The truncation of $f$ at $\delta$ is $\sum_{\gamma< \delta} f_\gamma t^\gamma$ and we shall denote it by $f|_\delta$.
We call $f$ \textbf{purely infinite}, \textbf{bounded}, \textbf{infinitesimal} if $\supp(f)\subseteq \Gamma^{<0}$, $\supp(f) \subseteq \Gamma^{\geq 0}$,  $\supp(f)\subseteq \Gamma^{>0}$, respectively.
We will distinctly name three components of $f$:
the \textbf{purely infinite part} $f|_0$, the \textbf{bounded part} $f_\preceq := f-f|_0$, and the \textbf{infinitesimal part} $f_\prec := f-f_\preceq$, so $f= f|_0 + f_{\preceq}, f_\preceq = f_0 + f_\prec$\\

\section{Valued Fields}

A \textbf{valued field} is a field $K$ equipped with a surjective map $v:K \rightarrow \Gamma\cup \{\infty\}$,  where $\Gamma$ is an additive ordered abelian group, such that for all $f,g \in K$ we have
\begin{enumerate}
	\item[(V0)] $v(f)= \infty \iff f=0,$
	\item[(V1)] $v(fg) = v(f) + v(g)$, 
    \item[(V2)] $v(f+g) \geq \min\{v(f),v(g)\}$
\end{enumerate}

Every valued field gives rise to:
\begin{enumerate}
	\item The valuation ring $\bigO:= \{f\in K: v(f)\geq 0\}$, which is a local ring,
    \item the maximal ideal $\smallo := \{f\in K:v(f)>0\}$ of $\bigO$, and
    \item the residue field $\smallk:= \bigO/\smallo$ of $K$.
\end{enumerate}

\begin{example}The canonical valuation on $\smallk((t^\Gamma))$ is given by the map $v:\smallk((t^\Gamma))^\times \rightarrow \Gamma$ where $v(f) = \min (\supp(f))$. We then observe that the corresponding valuation ring consists of the bounded elements of $\smallk((t^\Gamma))$, the maximal ideal of the valuation ring consists of the infinitesimal elements of $\smallk((t^\Gamma))$ and the residue field is isomorphic to $\smallk$. 
\end{example}

\subsection*{Other Structures in valued fields}
A monomial group $\fM$ of a valued field $K$ is a multiplicative subgroup of $K^\times$ such that for every $\gamma\in \Gamma$ there is a unique element $\fm\in \fM$ such that $v(\fm)= \gamma$.
\begin{example} 
Let $K = \smallk((t^\Gamma))$. Then the \textbf{canonical monomial group} of $K$ is the set $\{t^\gamma : \gamma \in \Gamma\}$.
\end{example}

\medskip \noindent
An \textbf{additive complement to the valuation ring} $\bigO$ of a valued field $K$ is an additive subgroup $V$ of $K$ such that $K = V\oplus\bigO$
\begin{example}
Let $K=\smallk((t^\Gamma))$. Then the \textbf{canonical additive complement} for $K$ is the set of purely infinite elements of $K$.
\end{example}

\section{The natural numbers}
\noindent
We start with the following well known result.
\begin{theorem}\label{Nund}
The theory of $(\N;+,\times)$ is undecidable.
\end{theorem}

\subsection*{Monadic Second-Order Logic} Given a structure $\cM=(M;\ldots)$, \textbf{monadic second-order logic of} $\cM$ extends first-order logic over $\cM$ by allowing quantification of subsets of $M$.
More precisely it amounts to considering the two-sorted structure $(M,\cP(M);\ldots,\in)$, where $\in \subseteq M\times \cP(M)$ has the usual interpretation.
The following Theorem and its proof appear in \cite{G}.
\begin{theorem}
The theory of $(\N,\cP(\N);\in)$ is decidable.
\end{theorem}

\begin{lemma}\label{timesinMOSNplus}
Multiplication on $\N$ is definable in $(\N,\mathcal{P}(\N);+, \in)$.
\end{lemma}
\begin{proof}
If the multiplication of consecutive numbers is defined, then general multiplication of two natural numbers can be defined in terms of addition:
\[n=mk \iff (m+k)(m+k+1) = m(m+1) + k(k+1) + n + n.\]

\noindent
If divisibility is defined, then multiplication of consecutive numbers is defined by
\[n=m(m+1)\iff \forall k(\in \N) (n|k \leftrightarrow [m|k \wedge (m+1)|k]).\]

\noindent
Divisibility can be defined using addition by
\[m|n \iff \forall S(\in \cP(\N)) (0\in S \wedge \forall x(\in \N) (x\in S \rightarrow x+m\in S) \rightarrow n\in S).\]
Since addition is a primitive, multiplication is defined in $(\N,\mathcal{P}(\N); +, \in)$.
\end{proof}

\begin{corollary}\label{MOSNplus}
The theory of $(\N,\mathcal{P}(\N); +, \in)$ is undecidable.
\end{corollary}
\begin{proof}
This follows from Lemma \ref{timesinMOSNplus} and Theorem \ref{Nund}.
\end{proof}

\section{Hahn Fields with Truncation}\label{HFWT}
Let $K=\smallk((t^\Gamma))$ be a Hahn field with non-trivial value group $\Gamma$. We consider $K$ as an $L$-structure where $L=\{0,1,+,\times,\bigO, \fM, V\}$, and the unary predicate symbols $\fM, \bigO,$ and $V$ are interpreted respectively as the canonical monomial group $t^\Gamma$, the valuation ring, and the canonical additive complement to $\bigO$. For $\gamma \in \Gamma$ and $\fm = t^\gamma$ we set $f|_\fm : = f|_\gamma$. Then we have the equivalence (for $f$, $v$ $\in  K$)

\[f|_1 = v \iff v\in V\ \&\ \exists g\in \bigO (f=v+g),\]
showing that truncation at $1$ is definable in the $L$-structure $K$. 
For $\fm \in t^\Gamma$ and $f\in K$ we have 
\[f|_\fm =  g \iff (\fm^{-1}f)|_0=\fm^{-1}g,\]
showing that the operation $(f,\fm)\mapsto f|_\fm : K\times t^\Gamma \rightarrow K$ is definable in the $L$-structure $K$. 

For convenience of notation we introduce the asymptotic relations $\preceq,\prec ,$ and $\asymp$ on $K$ as follows. 
For $f,g\in K$, $f\preceq g$ if and only if there is $h\in \bigO$ such that $f=gh$, likewise $f\prec g$ if and only if $f\preceq g$ and $g\not\preceq f$, and  $f\asymp g$ if and only if $f\preceq g$ and $g\preceq f$. 
Let $R := \{(\fm,f)\in t^\Gamma\times K: \fm\in t^{\supp(f)}\}$. Then $R$ is definable in the $L$-structure $K$ since for $a,b\in K$
\[(a,b)\in R\iff  a\in t^\Gamma \text{ and } b-b|_a \asymp a.\]

\begin{theorem}\label{NinK}
The $L$-structure $K$ interprets $(\N,\mathcal{P}(\N); +, \in)$.
\end{theorem}
\begin{proof}
Let $\approx$ be the definable equivalence relation on $K$ such that $f\approx g$, for $f,g\in K$, if and only if $\supp(f) = \supp(g)$. 
Take $\fn\in t^\Gamma$ such that $\fn\prec 1$. Consider the element $f= \sum_n \fn^n\in K$, and the set $S= \{g\in K: \supp(g) \subseteq \supp(f)\}$. 
Let $E\subseteq t^{\supp(f)}\times (S/\hspace{-1mm}\approx)$  be given by 
\[(\fm,g/\hspace{-1mm}\approx)\in E :\iff \fm \in t^{\supp(g)},\] 
and note that $E$ is definable in the $L$-structure $K$ since $R$ is. Define $\iota: \N \rightarrow t^{\supp(f)}$ by $\iota(n) = \fm^n$, and note that $\iota$ induces an isomorphism $(\N, \cP(\N);\in) \stackrel{\sim}{\longrightarrow} (t^{\supp(f)},S/\hspace{-1mm}\approx;E)$, such that $\iota(m+n) = \iota(m)\iota(n)$.
\end{proof}
\begin{corollary}
The theory of the $L$-structure $K$ is undecidable.
\end{corollary}
\begin{proof}
This follows easily from Theorem \ref{NinK} and Theorem \ref{MOSNplus}
\end{proof}

\subsection*{Defining the coefficient field $\smallk$} We now consider $K=\smallk(t^\Gamma))$ as an $L^-$-structure, where $L^-=\{0,1,+,\times, \bigO,V\}$. Note that for $f\in \bigO$ we have 
\[fV\subseteq V \iff f\in \smallk,\]
where we identify $\smallk$ with $\smallk t^0$. Thus we can define the coefficient field $\smallk$ in the $L^-$-structure $K$. 

\noindent
\textbf{Question:} Is it possible to define the monomial group $t^\Gamma$ in the $L'$-structure $K$?
\subsection*{An approach without the monomial group} Alternatively we may work in the setting of the two sorted structure $(K, \Gamma; v, T)$ where $K$ denotes the underlying field, $\Gamma$ is the ordered value group, $v$ is the valuation, and $T:K\times \Gamma \rightarrow K$ is such that $T(f,\gamma) = f|_\gamma$. Then we can define the binary relation $R\subseteq \Gamma\times K$ by 
\[(\gamma,f) \in R :\iff v(f-T(f,\gamma)) = \gamma .\]
We then obtain the following;
\begin{theorem}
The two-sorted structure $(K,\Gamma;v,T)$ interprets $(\N,\mathcal{P}(\N); +,\in)$.
\end{theorem}
The proof is similar to the proof of Theorem \ref{NinK}.
\begin{corollary}
The theory of the two-sorted structure $(K, \Gamma; v, T)$ is undecidable.
\end{corollary}
\section{Dividing lines in model theoretic structures}

We have already shown how $(\N;+,\times)$ can be interpreted in the $L$-structure $K$ and thus we know that it has the strict order property and the tree property of the second kind among others. In this section we make explicit a binary relation that witnesses these properties inside $K$. 
\subsection*{The independence property}
Let $L$ be a language and $\cM= (M;\ldots)$ an $L$-structure. We say that an $L$-formula $\phi(x;y)$ \textbf{shatters} a set $A\subseteq M^x$ if for every subset $S$ of $A$ there is $b_S\in M^y$ such that for every $a\in A$ we have that $M\models \phi(a;b_S)$ if and only if $a\in S$. Let $T$ be an $L$-theory. We say that $\phi(x;y)$ has the \textbf{independence property with respect to }$T$, or \textbf{IP} for short, if there is a model $M$ of $T$, such that $\phi(x;y)$ shatters an infinite subset of $M^x$. 

For a partitioned formula $\phi(x;y)$ we let $\phi^{opp}(y;x) = \phi(x;y)$, that is, $\phi^{opp}$ is the same formula $\phi$ but where the role of the parameter variables and type variables is exchanged.

\begin{lemma}\label{IPopp}
A formula $\phi(x;y)$ has $\IP$ if $\phi^{opp}$ has $\IP$.
\end{lemma}
\begin{proof}
By compactness the formula $\phi(x;y)$ shatters some set $\{a_J: J\in \cP(\N)\}$. Let the shattering be witnessed by $\{b_I:I\subseteq \cP(\N)\}$. Let $B= \{b_{I_i}: i\in\N\}$ be such that $I_i = \{Y\subseteq \N: i\in Y\}$. Then we have
\[\models\phi(a_J, b_{I_i}) \iff i \in J,\]
and thus $\phi^{opp}$ shatters $B$. 
\end{proof}

\subsection*{The Strict Order Property}
We say that a formula $\phi(x;y)$ has the Strict Order Property, or $\SOP$ for short, if there are $b_i\in M^y $, for $i\in \N$, such that $\phi(M^x,b_i) \subset \phi(M^x,b_j)$ whenever $i<j$. 

\begin{proposition}
The formula $\varphi(x;y)$, defining the relation $R$ as in section \ref{HFWT}, has $\SOP$.
\end{proposition}

\begin{proof}
Let $\Theta= \{\theta_i: i\in \N\}$ be any subset of $\Gamma$ such that $\theta_i <\theta_j$ for $i<j$, and consider the set $\{f_n = \sum_{i=0}^n t^{\theta_i}: i \in \N\}$. Note that $\varphi(K,f_m)\subset \varphi(K,f_n)$ for $m<n$.  
\end{proof}

\subsection*{The tree property of the second kind}
We say that a formula $\phi(x;y)$ has the \textbf{tree property of the second kind}, or $\TP$ for short, if there are tuples $b_j^i\in M^y$, for $i,j\in \N$, such that for any $\sigma:\N\rightarrow \N$ the set $\{\phi(x;b^i_{\sigma(i)}): i\in \N\}$ is consistent and for any $i$ and $j\neq k$  we have $\{\phi(x;b^i_j),\phi(x;b^i_k)\}$ is inconsistent.

\begin{lemma}\label{TP2IP}
If $\phi(x;y)$ has $\TP$ then $\phi$ has $\IP$.
\end{lemma}
\begin{proof}
Let $\{\phi(x,b^i_j)\}_{i,j\in \N}$ witness $\TP$ for $\phi(x;y)$. Fix $j$. Without loss of generality we will assume that $j=0$. Consider the set $\{b^i_0\}$. Let $I\subseteq \N$. By $\TP$ there is $a_I \in M^x$ such that \[\cM\models  \phi(a_I;b^i_j) \iff (i\in I \text{ and }j=0, \text{ or } i\notin I \text{ and } j=1).\]
Thus by Lemma \ref{IPopp} $\phi(x;y)$ has IP.
\end{proof}

\begin{lemma}\label{TP2Char}
Let $A=\{a_i:i\in \N\}\subseteq M^x$ and $B=\{b_I:I\in \mathcal{P}(\N)\}\subseteq M^y$. Assume that there is $\phi(x;y)$ such that for any fixed $b_I\in B$
\[\models \phi(a;b_I)\iff \text{ there is }  i\in I \text{ such that }a=a_i .\]
Then $\phi$ has $\TP$. 
\end{lemma}

\begin{proof}
Let $\phi$, $A$, and $B$ be as in the hypothesis of the Lemma. Let $P=\{p_i \in \N\}$ be the set of primes where $p_i\neq p_j$ for $i\neq j$. We construct $A^i_j\subseteq \N$ recursively as follows:
\begin{itemize}
	\item $A^0_j:= \{p_j^{n_0}:n_0>0\}$
    \item $A^i_j:= \{p_{n_i}^m: n_i\in \N,\ m\in A^{i-1}_{j}\}.$
\end{itemize}
So for example $A^1_2 =\{p_{n_1}^{p_2^{n_0}}: n_1,n_0>0\}$. \\
\textbf{Claim 1} For $\alpha \in \N^{n}$ we have that $\bigcap_{i<n} A^i_{\alpha(i)}\neq \emptyset$.\\
It is not hard to check that 
\[p_{\alpha(0)}^{\text{\reflectbox{$\ddots$}} ^{p_{\alpha(n-1)}}}\in \bigcap_{i<n} A^i_{\alpha(i)}.\]
\textbf{Claim 2:} For fixed $i$, and $j\neq k$ we have $A^i_j\cap A^i_k = \emptyset$. \\
For simplicity in notation we prove the case where $i=1$. Let $m\in A^1_j\cap A^1_k$. Then $m = p_{m_1}^{p_j^{m_0}} = p_{n_1}^{p_k^{n_0}}$. Since $p_j^{m_0}$ and $p_k^{n_0}$ are non-zero, we have that $p_{m_1} = p_{n_1}$, and thus $p_j^{m_0} = p_k^{n_0}$. Similarly, since $m_0$ and $n_0$ are non-zero we conclude that $p_j=p_k$, and thus $j=k$. 

Now let $b^i_j = b_{A^i_j}$. By compactness, together with Claim 1, we get that the set $\{\phi(x;b^i_{\sigma(i)}): i\in \N\}$ is consistent.
By the hypothesis of the Lemma, together with claim 2, we get that for any $i$ and $j\neq k$  we have $\{\phi(x;b^i_j),\phi(x;b^i_k)\}$ is inconsistent.
\end{proof}
 If $\phi(x;y)$ and $A$ are as in the lemma, we say that $\phi(x;y)$ and $B$ \textbf{only shatter} $A$ \textbf{in} $M$. Note that in this case $A$ is in fact a definable set.

\begin{proposition}\label{TP2forK}
	The formula $\varphi(x;y)$, defining the relation $R$ as in section \ref{HFWT}, has $\TP$.
\end{proposition}
\begin{proof}
Let $\Theta$ be a well-ordered subset of $\Gamma$ and consider the sets \[t^\Theta=\{t^\theta:\theta \in \Theta\}, \text{ and } B= \left\{\sum_{\delta\in \Delta} t^\delta: \Delta \subseteq \Theta\right\}.\] 
It is clear then that $\varphi(x;y)$ and $B$ only shatter $t^\Theta$, and thus by Lemma \ref{TP2Char} the formula $\varphi(x;y)$ has $\TP$.
\end{proof}

\begin{corollary}
The formula $\varphi(x;y)$, defining the relation $R$ as in section \ref{HFWT}, has $\IP$.
\end{corollary}
\begin{proof}
The result follows directly from proposition \ref{TP2forK} and lemma \ref{TP2IP}.
\end{proof}

\end{document}